\documentclass{svproc}
\usepackage{amsmath, amssymb}
\usepackage{mathrsfs}
\usepackage[usenames]{xcolor}
\usepackage{hyperref}

\newtheorem{thm}{Theorem}
\newcommand{\bt}{\begin{thm}}
\newcommand{\et}{\end{thm}}

\newtheorem{ex}[thm]{Example}

\newtheorem{cor}[thm]{Corollary}

\newcommand{\bc}{\begin{cor}}
\newcommand{\ec}{\end{cor}}

\newtheorem{lem}[thm]{Lemma}  

\newcommand{\bl}{\begin{lem}}
\newcommand{\el}{\end{lem}}

\newtheorem{prop}[thm]{Proposition}
\newcommand{\bp}{\begin{prop}}
\newcommand{\ep}{\end{prop}}

\newtheorem{defn}[thm]{Definition}

\newcommand{\bd}{\begin{defn}}   
\newcommand{\ed}{\end{defn}}

\newtheorem{rmrk}[thm]{Remark}   

\newcommand{\br}{\begin{rmrk}}
\newcommand{\er}{\end{rmrk}}

\newcommand{\VFto}{\stackrel {\mathcal{VF}}{\longrightarrow} }
\newcommand{\SIFto}{\xrightarrow{\mathcal{SIF} }}

\newcommand{\Ric}{\operatorname{Ric}}

\newcommand{\be}{\begin{equation}}
\newcommand{\ee}{\end{equation}}

\newcommand{\N}{\mathbb{N}}

\newcommand{\diam}{\operatorname{diam}}

\newcommand{\area}{\operatorname{Area}}
\newcommand{\vol}{\operatorname{Vol}}

\DeclareMathOperator{\tr}{tr}

\let\oldmarginpar\marginpar
\renewcommand\marginpar[1]{\-\oldmarginpar[\raggedleft\footnotesize #1]%
{\raggedright\footnotesize #1}}

\newcommand\bel[1]{\begin{equation}\label{#1}}
\renewcommand\ee{\end{equation}}

\newcommand{\Sig}{\Sigma}
\newcommand{\mS}{\mathcal{S}}
\newcommand{\mM}{\mathcal{M}}

\usepackage{url}

\usepackage{graphicx}

\begin{document}
	\mainmatter              
	\title{Lorentzian area and volume estimates for integral mean curvature bounds}
	\titlerunning{Lorentzian area and volume estimates for integral mean curvature bounds}  
	%
	\author{Melanie Graf\inst{1} \and Christina Sormani\inst{2}}
	\authorrunning{Melanie Graf and Christina Sormani} 
	\institute{University of Tübingen\\
		\email{graf@math.uni-tuebingen.de} 
		\and
	Lehman College and CUNY Graduate Center\\
		\email{sormanic@gmail.com}}
	
	\maketitle

\begin{abstract} 
	In the present paper we establish area and volume estimates for spacetimes satisfying the strong energy condition in terms of the area  and the $L^n$-norm of the second fundamental form or the mean curvature of an initial Cauchy hypersurface. We believe that these estimates will lay some of the groundwork in establishing new convergence results for Cauchy developments $(M_j,g_j)$ of suitably converging initial data $(\Sigma_j,h_j,K_j)$.

\keywords{Lorentzian Geometry, volume estimates, integral mean curvature bounds}
\end{abstract}

\section{Introduction}

In 1952, Choquet-Bruhat proved that given an initial data set, $(\Sigma, h, K)$, 
consisting of a Riemannian manifold $(\Sigma, h)$ and a $(0,2)$ symmetric tensor field $K$ satisfying consistency conditions, there is a future Cauchy development $(M, g)$ where $M$ is a globally hyperbolic Lorentzian manifold satisfying the vacuum Einstein Equations with Cauchy hypersurface $\Sigma$
such that $h=g$ restricted to $\Sigma$ and $K$ is the second fundamental form of $\Sigma$ \cite{Choquet-Bruhat-52}.  See Figure~\ref{fig-init-data}.    In joint work with Geroch, she proved
there is a unique maximal development $(M,g)$ which contains all Cauchy developments \cite{CBG:69}.    Examples of such developments are the Minkowski and Schwarzschild spacetimes whose initial data sets are Euclidean and Riemannian Schwarzschild space both with $K=0$.  Similar results guaranteeing the existence of a unique maximal Cauchy development for given initial data have since been established for a variety of matter models. Examples of such developments  with compact initial data sets using a perfect-fluid matter model are the
FLRW spacetimes which are applied to model the cosmos and predict the big bang.

\begin{figure}
	\begin{center}
		\includegraphics[width=0.4\textwidth]{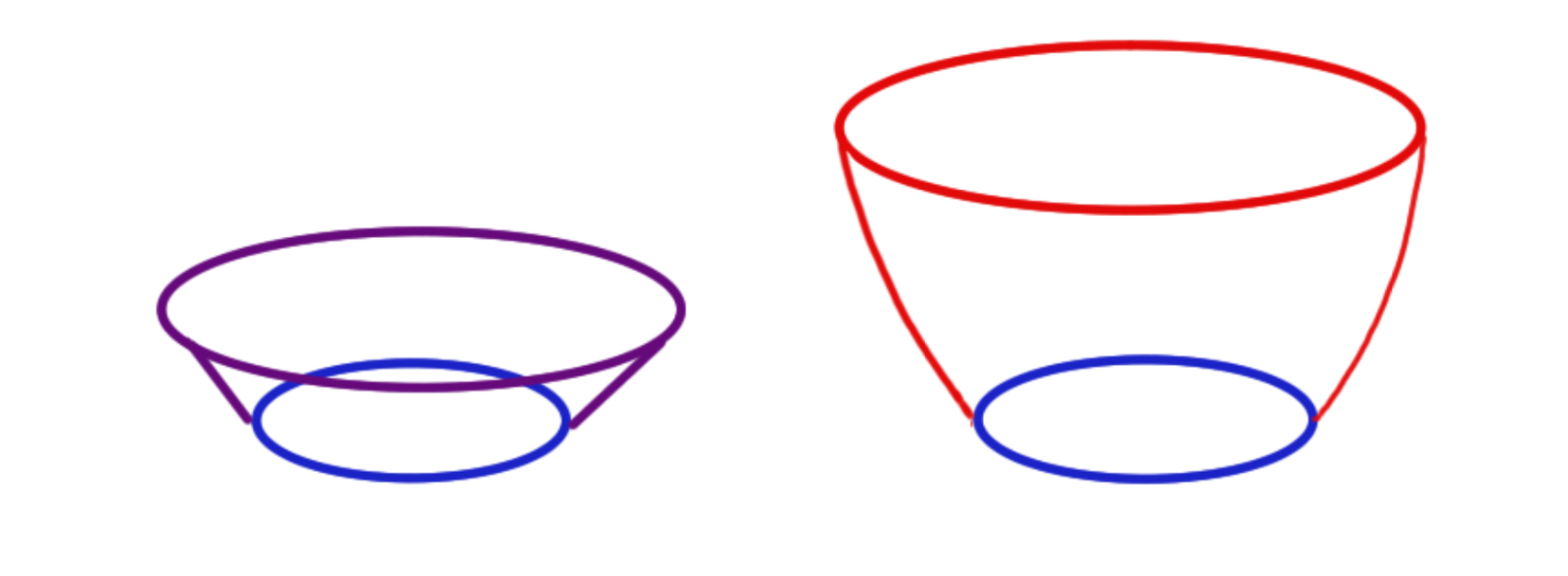}
	\end{center}
	\caption{On the left we see $\textcolor{blue}{(\Sigma, h)}$ and $\textcolor{violet}{K}$, and on the right we see  $\textcolor{red}{(M,g)}$.} \label{fig-init-data}
\end{figure}

It has been of great interest to understand the stability of these maximal developments under variations of the initial data. 
If one assumes
\be \label{init-data-conv}
(\Sigma_j, h_j, K_j) \to  (\Sigma_\infty, h_\infty, K_\infty)
\ee
in some sense then can one prove
\be \label{spacetime-conv}
(M_j, g_j) \to (M_\infty, g_\infty)
\ee
in some sense?  Ideally one would like to consider fairly weak convergence
of initial data $(\Sigma_j, h_j, K_j)$ to $(\Sigma_0,h_0,K_0)$
which could arise naturally in cosmological
scenarios: where the $\Sigma_j$ might have increasingly many increasingly
thin gravity wells and black holes.     See Figure~\ref{fig-sequence}.

\begin{figure}
	\begin{center}
		\includegraphics[width=0.8\textwidth]{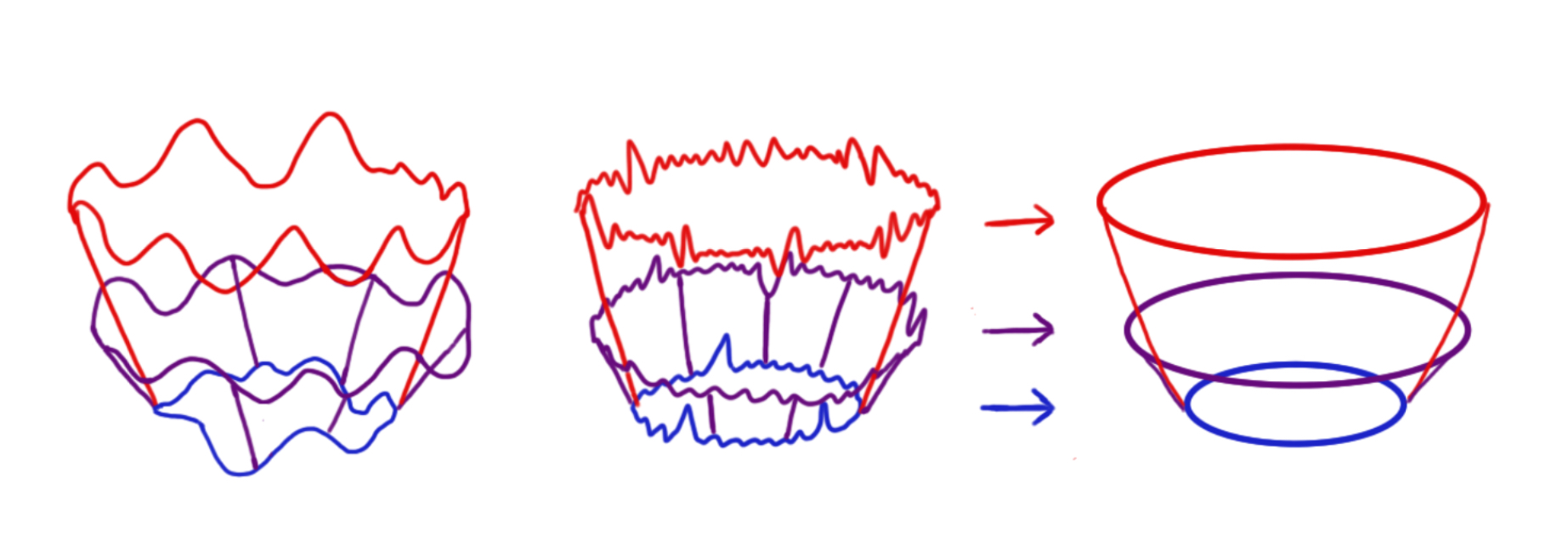}
	\end{center}
	\caption{Here $(\textcolor{blue}{\Sigma_j, h_j}, \textcolor{violet}{K_j}) 
		\to  (\textcolor{blue}{\Sigma_\infty, h_\infty}, \textcolor{violet}{K_\infty})$
		and $\textcolor{red}{(M_j, g_j)} \to \textcolor{red}{(M_\infty, g_\infty)}$ 
		where we have only weak not smooth convergence of the initial data and Cauchy developments.} \label{fig-sequence}
\end{figure}

Significant work has been conducted assuming that $(\Sigma_\infty, h_\infty, K_\infty)$ is
Euclidean space with $K_{\infty}=0$ and that $\Sigma_j$ are diffeomorphic to $\Sigma_\infty$ and that the 
$h_j$ converge at least in the $C^0$ sense to $h_\infty$ with various bounds on the $K_j$.    See the work of
Christodoulou, Klainerman, Rodniansky,
Bieri and Szeftel \cite{ChKl}\cite{KlRod}\cite{Bieri}\cite{KlRodSz}.   There has also been
recent work on the stability of Schwarzschild Initial data by Dafermos, Holzegel, Rodnianski and Taylor \cite{DHR}\cite{DHRT}. 
There is work of
Moncrief, Hao, Chrusciel, Berger, Isenberg, and Ringström on the stability of FLRW spacetimes
surveyed in Chapter XVI of \cite{Choquet-Bruhat-OxMono}.   However, if one is going to use FLRW models to predict the behavior of the cosmos, one cannot assume that the initial data sets are smoothly close to the homogeneous and isotropic initial data used in these models.   In particular, it has been observed that the universe contains black holes, and thus the spacelike universe is not even diffeomorphic to one of these models.  Dyer and Roeder have constructed swiss cheese models of the universe \cite{Dyer-Roeder} which are used numerically but very little theoretical mathematics research has been done to study how close such models are to the FLRW spacetimes.    
See work of the second author in
\cite{Sormani-GAFA} for a discussion of the difficulties even for the spacelike slices.

The volume preserving intrinsic flat ($\mathcal{VF}$) distance has been successfully applied to produce a small distance between Riemannian manifolds that have thin wells and tiny holes and ones where the wells and holes have disappeared.   
First defined by the second author with Wenger in \cite{SW-JDG} applying work of Ambrosio-Kirchheim \cite{AK}, the intrinsic flat $\mathcal{F}$ distance is closely related to the flat distance used in Geometric Measure Theory but it is intrinsic in the sense that it does not need an external space containing the two manifolds.   The second author, Lakzian, Allen, and Perales have proven in a series of papers that one can control the $\mathcal{VF}$ distance between spaces using volume and area estimates  \cite{LS13} \cite{VADB} \cite{AP-VADB}.  These results have been applied to spacelike stability questions arising in General Relativity by both authors, Huang, Lee, LeFloch, Stavrov, Perales, Cabrera-Pacheco, and others (cf final section of the survey \cite{Sormani-scalar-21}).

This paper has been written to develop the essential volume estimates needed to help achieve the spacetime intrinsic flat  ($\mathcal{SIF}$) notion of convergence.   The goal would be to prove one day that if $(\Sigma_j, h_j) \to (\Sigma_\infty, h_\infty)$ in the $\mathcal{VF}$ sense  and
if $K_j \to K_\infty$ converges in some natural weak sense, then compact regions in the global maximal developments 
will converge in the $\mathcal{SIF}$ sense.   Since $(\Sigma_j, h_j) \to (\Sigma_\infty, h_\infty)$ in the $\mathcal{VF}$ sense implies $\area_{h_j}(\Sigma_j) \to \area_{h_\infty}(\Sigma_\infty)$ and work by the first author with Bryden and Khuri in \cite{BKS21} indicates that an $L^p$ bound on the $K_j$ is natural, in this paper we assume related conditions and work to draw consequences about the volumes and areas in $(M_j, g_j)$.

It should be noted that the notion of $\mathcal{SIF}$ convergence first arose in discussions between the second author, Shing-Tung Yau, and Lars Andersson: where it was proposed that one use the Andersson-Galloway-Howard cosmological time function of \cite{cosmo-time} combined with $\mathcal{F}$ convergence \cite{SW-JDG}.   See \cite{Sormani-Oberwolfach-18} for the plan and work of Vega, Sakovich, Allen, and Burtscher in \cite{Sormani-Vega-1},  \cite{AllenBurtscher1}, \cite{Vega21}, \cite{Sakovich-Sormani-SIF} for progress.   However, one does not need to read these papers or know anything about intrinsic flat convergence to read this paper, as here we are focussing purely on obtaining the volume and area estimates needed for this program but are not yet applying them.

In this paper we estimate the volume bounds of spacetimes $M$ satisfying the strong energy condition in terms of
controls on their initial data sets, $(\Sigma,h,K)$.   We do not assume these $M$ satisfy the Einstein equation with any particular matter model.   We use the cosmological
time functions of Anderson-Galloway-Howard \cite{cosmo-time} (cf. Definition~\ref{def:cosmotime}) defined by
\be
\tau(q):=\sup_{p\in M,\, p\leq q}\tau(p,q).
\ee
In particular we prove:

\begin{thm} \label{main-thm} Let   $(\Sigma ,h , K)\subset (M,g)$ be a compact initial data set for a future Cauchy development $(M,g)$. 
If $(M,g)$ satisfies the strong energy condition, then
	\begin{align} 
|\mathcal{S}_t  | &\leq 
2^{n-1}  \left(t^n \, ||K||_{L^n}^n + {|\Sigma |} \right)\\
|\Omega_t  | &\leq 
2^{n-1}  \left(\frac{t^{n+1}}{(n+1)} \, ||K||_{L^n}^n + t {|\Sigma |} \right)
\end{align}
where $n=\dim(\Sigma)$, $\mS_t:=\tau^{-1}(t)\setminus \mathrm{Cut}(\Sigma)$ and $\Omega_t:=\tau^{-1}([0,t])$.
\end{thm} 

In Theorem~\ref{thm:areaest} and Corollary~\ref{cor:volest}, we prove similar results for regions evolving from relatively compact subsets of $\Sigma $ (see Lemma~\ref{lem:subsets}). We further show that these estimates continue to hold if $\Sigma$ (resp.~$A\subseteq \Sigma$) is non-compact (resp.~not relatively compact) but still has finite area, see Corollary~\ref{cor:noncompact} and Remark \ref{rem:noncompactA} within.     We exclude the cut locus from the $\tau$-level sets when estimating the areas to avoid issues with defining the area of a non-smooth set in a Lorentzian manifold.   In Section~\ref{sec:limitareas} we define a generalized area for the full level sets imitating methods from Riemannian geometry. 
We obtain analogous bounds for these generalized areas and volumes in Lemma~\ref{Vol-Sigma-T-1} and Corollary~\ref{Vol-Sigma-T-2}.

A crucial aspect of our theorem is that we only require $L^n$ bounds on the second fundamental form/mean curvature.
In Section~\ref{Sect-Example} we present an example indicating that weaker assumptions on the second fundamental form cannot be applied to control the volumes of regions.    Prior theorems controlling these volumes in the Lorentzian setting have been proven by Grant-Treude and Treude \cite{TG}\cite{TreudeDiploma}, and later extended to spacetime metrics which are only $C^{1,1}$ in  \cite{Graf-C1.1-vol-comp}, but those papers required uniform pointwise bounds on the second fundamental form/mean curvature.  See Remark~\ref{Comp-to-GT} within.   Naturally these types of volume estimates and their consequences were first pioneered in the Riemannian setting, starting with Bishop and Gromov \cite{Bishop},\cite{Gromov}.   Perales and Paeng \cite{Perales},\cite{Paeng} found volume bounds using integral mean curvature bounds in the Riemannian setting, and indeed we were directly inspired by Perales' work.

An immediate corollary of our theorem is that if we consider a sequence of  compact initial data sets 
$(\Sigma_j ,h_j , K_j)$ for future Cauchy developments $(M_j,g_j)$ satisfying the strong energy condition, then weak convergence of the initial data implies boundedness of the volumes.  More precisely:
 If 
 \be 
 (\Sigma_j, h_j) \VFto (\Sigma_{\infty},h_{\infty})
 \ee
  and 
  \be 
  ||K_j||_{L^n(\Sigma_j)}\to ||K_{\infty}||_{L^n(\Sigma_{\infty})},
  \ee
   then the level sets of $\tau$ have bounded (generalized) areas as described in 
   Corollary~\ref{Vol-Sigma-T-2}. Thus it should be easy to show $\tau^{-1}(t)$ have subsequences 
which converge in the
intrinsic flat sense even if they are not smooth Riemannian manifolds using the theory in \cite{SW-JDG}.  We conjecture that 
\be
\tau_j^{-1}[0, T] \subset M_j \,\,\, \SIFto \,\,\,\tau_{\infty}^{-1}[0, T] \subset M_{\infty}
\ee
when $M_\infty$ is an FLRW spacetime and $M_j$ satisfy the Einstein Equation with an appropriate matter model.  This would then answer the question 
as to why one can use FLRW spacetimes to model the evolution of our universe.   

\vspace{.3cm}

\noindent
{\bf Acknowledgements}: We would like to thank Lars Andersson and Shing-Tung Yau for first suggesting that a spacetime intrinsic flat convergence defined using the cosmological time might be applied to solve the open questions described in the introduction.  We'd like to thank James Grant for introducing us to one another and discussing key ideas with us at the beginning of this project.

\section{Background}

{\bf Standard notations and conventions.} We will use the following standard definitions and notations from Lorentzian geometry: A piecewise smooth curve $\gamma :I\to M$ is called future directed timelike/causal/null if $\dot{\gamma}(t)$ is future pointing timelike/causal/null for all $t$. 
We write $p\ll q$ if there exists a future directed (f.d.) timelike  curve from $p$ to $q$ (and $p\leq q$ if there exists a future directed causal curve from $p$ to $q$ or $p=q$) and set 
\begin{align}
I^{+}(p): & =\left\{ q\in M:\, p\ll q\right\} \\
J^{+}(p): & =\{q\in M:\, p\leq q\}.
\end{align}
 For $p\in M$ we define the (future) time separation or {\em Lorentzian distance} of a point $q\in M$ to $p$ by
	\begin{equation}
	\tau_p(q)\equiv\tau(p,q):=\sup\left\{ L(\gamma):\gamma\,\text{is a f.d. causal curve form }p\text{ to }q\right\} \cup\{0\}.\label{eq:point time sep}
	\end{equation}
Similarly one defines the future time separation to a set $A\subseteq M$
	by
	\begin{equation}
	\tau_{A}(q):=\sup_{p\in A}\tau(p,q).\label{eq:subset time sep}
	\end{equation}
For a spacelike hypersurface $\Sigma \subseteq M$ with future pointing unit normal vector field $\mathbf{n}$ we define the second fundamental form by 
\be II(V,W):=-g(\mbox{nor}(\nabla_VW), \mathbf{n})\ee for vector fields $V,W$  tangential to $\Sigma$ 
and the scalar mean curvature $H$ by \be H:=\mathrm{tr}_\Sigma K =\tr_\Sigma \mathbf{S},\ee
where $\mathbf{S}$ is the the shape operator for $\Sigma$ defined as $\mathbf{S}: X\mapsto \nabla_X \mathbf{n}$.

\subsection{Our setting}

We consider smooth initial data sets $(\Sigma ,h , K)$, where $(\Sigma ,h)$ is an $n$-dimensional Riemannian manifold 
and $K$ is a symmetric $(0,2)$-tensor field on $\Sigma$, with a {\em future globally hyperbolic development} or {\em future Cauchy development} $(M,g )$. That is, we assume that 
\begin{itemize}
	\item[(i)] $(M,g)$ is a time-oriented $(n+1)$-dimensional Lorentzian manifold $(M,g)$ with spacelike boundary $\partial M=\Sigma $,
	\item[(ii)] $h$ is the metric induced by $g $ on $\Sigma $ and $K$ is the second fundamental form 
	\item[(iii)] and $(M, g)$ is globally hyperbolic with Cauchy hypersurface $\Sigma$ such that additionally $M=J^+(\Sigma)$.
\end{itemize} 

Note that we do not require any specific matter model. We will be interested in the future evolution of $(\Sig,h,K)$ within a future Cauchy development given certain mean curvature and Ricci curvature bounds.

\br[Our curvature assumptions] \normalfont 
For the Ricci curvature we will require that 
\be \Ric(\dot{\gamma},\dot{\gamma})\geq 0 \label{eq:Ricbound} \ee
along any future directed timelike geodesic $\gamma$ starting orthogonally to $\Sigma$.  
Note that this bound is implied by the {\em strong energy condition}, a common energy condition in general relativity. In its physical formulation the strong energy condition imposes
\be
T(X,X)-\frac{1}{2} \mathrm{tr}(T) g(X,X)\geq 0 \quad\quad \forall\;\; \mathrm{timelike}\;\; X\in TM,
\ee
where $T$ denotes the stress-energy tensor. Via the Einstein equations (without cosmological constant) this condition on $T$ is equivalent to the geometric curvature condition
\be\Ric(X,X)\geq 0 \quad\quad \forall\;\; \mathrm{timelike}\;\; X\in TM\ee
which often is itself called the strong energy condition.
Clearly, this implies \eqref{eq:Ricbound}.

We postpone the discussion about the mean curvature bounds to Remark \ref{rem:meancurvrem}.
\er

We remark that, while we in general  a priori assume both the existence of the future Cauchy development $(M,g)$ and the strong energy condition (instead of, e.g., restricting ourselves to certain matter models), in vacuum and many other matter models it is well established that any initial data satisfying the constraint equations has a unique maximal globally hyperbolic development:

\br[Vacuum initial data]
In \cite{Choquet-Bruhat-52} Choquet-Bruhat proved that any initial data set $(\Sig,h , K)$ with $h \in M_s$ and $k\in H_{s-1}$, $s>\tfrac{n}{2}+1$, which satisfy the vacuum Einstein constraints,
\begin{align}
\mathrm{R}(h)-|K|_h^2+(\tr K)^2&=0\\
\mathrm{div}_h(K-(\tr K) h)&=0,
\end{align}
admits a globally hyperbolic vacuum Einstein development $(M,g)$ as above.
In joint work with Geroch \cite{CBG:69}, she proved this development is unique if it is taken to be maximal.
\er

\subsection{Comparison spaces}\label{sec:compspaces}

Similarly to the estimates in \cite{TG}, at various points in our analysis it will be interesting to compare our bounds to the respective quantities in a model spacetime which we will define in this section (following \cite[Sec.~4.2]{TG}). While these model spaces are well known, we want to present a quick derivation here. 
We start by making an ansatz as a warped product spacetime $M=I\times (N^n_c, h_c)$, where $(N^n_c, h_c)$ denotes the $n$-dimensional, simply-connected (Riemannian) manifold of constant curvature $c \in \{-1,0,1\}$, with metric
\bel{FRWmetric}
g = - dt^2 + a(t)^2 h_c,
\ee
where $a:I \to (0,\infty) $ is a smooth function.
We now wish to determine model spacetimes $M_\beta$ satisfying that $\Ric=0$ and $\Sigma_0:=\{0\}\times (N^n_c, h_c)$ is a Cauchy hypersurface with constant mean curvature $\beta\geq 0$.

For $X,Y$ tangential to $N^n_c$ we have (cf.~\cite[Cor.~12.10, Cor.~10.43]{ON})
\begin{align*}
\Ric(\partial_t, \partial_t) &= - n \frac{a''}{a},\\
\Ric(\partial_t,X) &=0,\\
\Ric(X,Y)&= \big( (n-1)(\frac{a'}{a})^2+\frac{(n-1)c}{a^2}+\frac{a''}{a}\big) g(X,Y).
\end{align*}
For the mean curvature of a constant-$t$-slice $\Sigma_t=\{t\}\times N_c^n$ we get (using \cite[Cor.~12.8(3)]{ON})
\be
H = -\mathrm{tr}_{\Sigma_t}\, g(K(.,.),\mathbf{n})= \tr_{\Sigma_t} (\frac{a'}{a} g) =\frac{na'}{a}=\frac{n}{a}.
\ee
So, for $c=-1$ (i.e., hyperbolic time slices) 
and \be a_\beta(t):=t+\frac{n}{\beta},
\ee
$(M,g)$ has $\Ric =0$ and $H_0=\beta $ as desired.

Summing up we have arrived at the following definition.

\begin{defn}\label{def:modelgeom}
We define the model geometry $(M_\beta, g_\beta)$ for an initial mean curvature $\beta \geq 0$ as the spacetime $M_\beta = [0,\infty) \times \mathbb{H}^n$ with Lorentzian metric 
\be g_\beta = -dt^2+(t+\frac{n}{\beta})^2 h_{-1},
\ee
where $h_{-1}$ denotes the standard Riemannian metric on $\mathbb{H}^n$, and initial Cauchy hypersurface
\be
\Sigma_{\beta, 0} := \{ 0 \} \times \mathbb{H}^n.
\ee
\end{defn}

Note that the areas and volumes of the time evolution of any  $A\subseteq \Sigma_{\beta,0}$ (with finite area) in the model spaces may be directly calculated as
\begin{align}\label{eq:comparisonarea}
|\Sigma_{\beta, t}(A)| &=|\{t\}\times A|= \frac{a(t)^n}{a(0)^n} |A| = |A| \left( 1+ \frac{\beta t}{n} \right)^n,
\\
|\Omega_{\beta, t }(A)| &= \int_0^t |(\Sigma_{\beta, s})(A)| \, ds =
|A| \int_0^t \left( 1+ \frac{\beta s}{n} \right)^n \, ds,
\end{align}

where we use $|.|$ to denote either the $(n+1)$-dimensional spacetime volume or the $n$-dimensional induced area depending on if we look at a spacetime region or a subset of a time slice/spacelike hypersurface.

\subsection{The cosmological time function and its properties}

Going back to non-model spacetimes, in order to study the time evolution of $\Sigma $ we need to equip $M$ with a ''global time'' for which $\Sigma $ equals the ''time $=0$'' level set and the ''time $=T$'' levels may be interpreted as the future of $\Sigma $ at time $T$. One possibility is to use the future time separation to the initial Cauchy hypersurface, $\tau_\Sigma $, corresponding to the Lorentzian distance to the initial Cauchy hypersurface $\Sigma=\partial M$. Another option for defining a global time is the cosmological time function introduced in \cite{cosmo-time} (note that  the authors only consider spacetimes without boundary but their definition straightforwardly extends to our setting):

\begin{defn}[Cosmological time]\label{def:cosmotime}
The cosmological time $\tau :M\to [0,\infty]$ is defined as
\begin{align}
\tau(q):=\sup_{p\leq q}\tau(p,q) 
\end{align} 
\end{defn}

We now show that the cosmological time agrees with the Lorentzian distance to $\Sigma $. In the proof we use that, because $\Sigma $ is a Cauchy hypersurface, $\tau_{\Sigma}$ is finite valued and continuous and that for any $p\in I^+(\Sigma)=J^+(\Sigma)\setminus \Sigma$ there exists $q\in\Sigma$ and a future directed timelike curve $\gamma: [a,b]\to M$ from $q$
to $p$ such that
\be\label{maxgeod}
L(\gamma|_{[a,t]})=\tau_{\Sigma}(\gamma(t))=\tau(q,\gamma(t)) \quad\quad \forall t\in [a,b].
\ee
Any such maximizing
curve $\gamma$ has to be a (reparametrization of) a timelike geodesic starting orthogonally to $\Sigma$ and is unique up to parametrization. These are standard results from Lorentzian geometry (see, e.g., \cite[Sec.~14]{ON}, \cite[Thm 3.2.19, Cor 3.2.20]{TreudeDiploma}).

\begin{lem}\label{lem:cosmotime=TGtime}Let $(M, g)$  be a gobally hyperbolic spacetime  with Cauchy hypersurface $\Sigma=\partial M$ and $M=J^+(\Sigma)$. Then $\tau=\tau_\Sigma$.
\end{lem}
\begin{proof} Let $q\in M$. Clearly,
\be
\tau_\Sigma(q) =\sup_{p\in\Sigma } \tau(p,q)= \sup_{p\in\Sigma, p\leq q } \tau(p,q)\leq  \sup_{p\leq q } \tau(p,q)= \tau(q).
\ee
 On the other hand, for any $p\leq q$ with $p\in M\setminus \Sigma = J^+(\Sigma)\setminus \Sigma = I^+(\Sigma)$ there exists $p_\Sigma \in \Sigma $ with $p_\Sigma \ll p\leq q$ and hence 
 \be
\tau(p_\Sigma  ,q) >\tau(p,q).
 \ee This shows $
\tau_\Sigma(q) = \sup_{p\in\Sigma, p\leq q } \tau(p,q)\geq  \sup_{p\leq q } \tau(p,q)= \tau(q).$ \qed
\end{proof}

Using $\tau $ we define level sets 
\be
\Sigma_t:=\tau^{-1}(t)
\ee
of $\tau $ as well as the $(n+1)$-dimensional spacetime evolution until time $t$ starting from $\Sigma $ 
\be
\Omega_{t}:=\tau^{-1}([0,t]).
\ee
One of our goals is to adapt estimates from Treude-Grant, \cite{TG}, to estimate $n$-dimensional ''limit-areas'' of $\Sigma_t$ (see section \ref{sec:limitareas}; note that $\Sigma_t$ isn't necessarily a smooth spacelike hypersurface) and $(n+1)$-dimensional spacetime volumes of $\Omega_t$ 
while assuming only a bound on the timelike Ricci curvature and on the $L^n$-norm of the mean curvature $H$ of $\Sigma =\tau^{-1}(0)$. 

To do this we need some additional background about properties of $\tau\equiv \tau_\Sigma$, in particular concerning its level sets and smoothness away from a spacetime measure zero set -- the Lorentzian cut locus of $\Sigma $. To reduce the amount of additional background needed we won't present the following results in their most general form.
We will also omit all proofs, referring instead mainly to \cite{TG,TreudeDiploma}.\footnote{Note that, while these references look at Lorentzian manifolds without boundary and thus don't technically treat our case where the Cauchy hypersurface is the boundary of $M$, all arguments work the same in our case.} Additionally a comprehensive study of the Lorentzian cut locus for the Lorentzian distance to a point may be found in \cite[Sec.~9]{BEE} and many arguments therein are completely analogous to our case where one considers the distance to a Cauchy hypersurface instead.

 Given a spacelike Cauchy hypersurface $\Sigma $ in a spacetime $M$, we denote the future unit normal bundle of $\Sigma $ by $N^+\Sigma $. i.e.,	
\be	N^+\Sigma:=\left\{ v\in T\Sigma^\perp:  v\, \text{future pointing}, \, g(v,v)=-1\right\} 
\ee
Analogous to the Riemannian case we now define the cut function and the cut locus:

\begin{defn}
	[Cut function]Let $\left(M,g\right)$ 
	be globally hyperbolic and $\Sigma\subset M$ be a spacelike Cauchy hypersurface. The function 
	\begin{align}
	s_{\Sigma}:\, &N^+\Sigma  \to [0,\infty ]\\
	s_{\Sigma}(v) & :=\sup\left\{ t>0:\,\tau_{\Sigma}(\gamma_{v}(t))=L(\gamma_{v}|_{\left[0,t\right]})\right\},
	\end{align}
	where $\gamma_v$ denotes the unique geodesic with initial value $\dot{\gamma}(0)=v$, is called the (future) cut function. 
\end{defn}

\begin{defn}
	[Cut locus]The tangential (future) cut locus is defined as
	\be
	\mathrm{Cut}_{T}(\Sigma):=\left\{ s_{\Sigma}(v)v:\, v\in N^+\Sigma\:\:\mathrm{and}\,s_{\Sigma}(v)v\in\mathcal{D}\right\} \subset T\Sigma^\perp,
	\ee
	where $\mathcal{D}$ denotes the domain of definition of the exponential map. The (future) cut locus is defined as the image of the tangential
	cut locus under the exponential map:
	\be
	\mathrm{Cut}(\Sigma):=\exp(\mathrm{Cut}_{T}(\Sigma)).
	\ee
\end{defn}

Note that from these definitions it immediately follows that all points on a maximizing geodesic as in \eqref{maxgeod} {\em before} the endpoint will {\em not} belong to the cut locus.

Similar to \cite{TG} our goal is to use (a Lorentzian version of) Laplacian comparison to estimate areas of level sets of $\tau \equiv \tau_\Sigma $. Unfortunately, as in the Riemannian case, the Lorentzian distance function is not smooth on all of $M$ and the best one can hope for is smoothness away from $\mathrm{Cut}(\Sigma)$. In the case of an $n$-dimensional Riemannian manifold the cut locus is closed and has $n$-dimensional measure zero because the cut function is Lipschitz continuous. 
In the Lorentzian setting one does not know if the cut function is continuous,  
but we do at least have that $v\mapsto s_\Sigma(v)$ is lower semi-continuous and  and continuous at each point $v\in N^+\Sigma$ where either $s_\Sigma(v)=\infty $ or $v s_\Sigma(v) \in \mathcal{D}$ (\cite[Prop.~3.2.29]{TreudeDiploma}, see also \cite[Prop.~9.5, 9.7]{BEE}).

From this and the local Lipschitz continuity of the normal exponential map we immediately get that
\begin{enumerate}
	\item the tangential cut locus is closed and has $(n+1)$-dimensional measure zero in $(T\Sigma)^\perp $ and
	\item the cut locus $\mathrm{Cut}_\Sigma$ is closed and has $(n+1)$-dimensional measure zero.
\end{enumerate}
One can further show 
\begin{enumerate}
	\item[(3)] $\mM:=M\setminus \mathrm{Cut}(\Sigma)$ is open in $M$ and diffeomorphic to \[\{tv: v\in N^+\Sigma \,\mathrm{and}\, t\in [0,s_\Sigma (v)\}\subseteq (T\Sigma)^\perp\] via $\exp $ (\cite[Thm.~3]{TG},\cite[Thm.~3.2.31]{TreudeDiploma}),
	\item[(4)] the cosmological time function $\tau  $ is smooth on $\mM$ with past directed unit timelike gradient (\cite[Prop.~1]{TG}),
	\item[(5)] the set $\mS_t:=\Sigma_t\setminus \mathrm{Cut}(\Sigma)=\mM \cap \tau^{-1}(t)$ is a smooth spacelike hypersurface in $M$ (\cite[Sec.~2.5]{TG}).
\end{enumerate}

As a spacelike hypersurface $\mS_t$ is an $n$-dimensional Riemannian manifold with metric $h_t:=g|_{T\mathcal{S}_t\times T\mathcal{S}_t}$ induced by $g$. It therefore has the  well defined area (where we use "area" to mean the induced $n$-dimensional volume)
\be
|\mS_t|\equiv \area(\mS_t) =\int_{\mS_t} d\mu(h_t),
\ee
where $d\mu(h_t)\equiv d\mu_t$ denotes the Riemannian volume measure corresponding to $h_t$.

Instead of looking at the time evolution of $\Sigma $ as a whole we will sometimes want to only look at the evolution of a subset $U\subseteq \Sigma $. To define $\mS_t(U)\subseteq \mS_t$, note that any point $p$ in $\mM$ lies on a {\em unique} geodesic maximizing the distance to $\Sigma$ and thus has a unique base point $p_\Sigma \in \Sigma$ for which $\tau(p)=\tau(p_\Sigma, p)$.
  We define
\be
\mS_t(U):=\{p\in \mS_t: p_\Sigma \in U \}=\exp(\mathcal{D}\cap (T\Sigma)^\perp|_U)\cap \mS_t.
\ee
From the second description we see that $\mS_t(U)\subseteq \mS_t$ is open in $\mS_t$ if $U$ is open in $\Sigma$. This can be used to argue that if $U\subseteq \Sigma $  is Borel measurable, then so is $\mS_t(U)$ and 
\be
|\mS_t(U)|\equiv \area(\mS_t(U))=\int_{\mS_t(U)} d\mu_t=\sup_{\mathrm{compact}\,K\subseteq \mS_t(U)} |K|.
\ee

Regarding the $(n+1)$-dimensional spacetime volumes note that, because the cut locus has $(n+1)$-dimensional measure zero, we have
 \be
 |\Omega_t|\equiv \vol(\Omega_t)=\vol(\Omega_t\cap \mM)=\int_0^t |\mS_\tau| d\tau.
 \ee
So for the spacetime volumes of $\Omega_t$ it doesn't make a difference whether we work within $\mM$ or $M$ and by the co-area formula this volume can be computed by integrating the areas of $\mS_\tau$.

\section{Area and Volume Estimates}\label{GTVol}

In this section we will 
we will derive estimates for the areas of the smooth $n$-dimensional Riemannian submanifolds $(\mathcal{S}_t, g|_{T\mathcal{S}_t\times T\mathcal{S}_t})$ assuming only a bound on the area $|\Sigma |$ of our initial data and a bound on the $L^n$-norm of the mean curvature of $\Sigma $ (or even more precisely, a bound on the average of $(H_+)^n$, where $H_+$ denotes the positive part of $H$). This then leads to estimates on the spacetime volume of $\Omega_t$, see Corollary \ref{cor:volest}.

\subsection{Basic area and volume estimates using integral mean curvature bounds}

The area and volume comparison results of~\cite{TG} may be generalized as follows.

\begin{thm}
\label{thm:areaest}
Let $(M, g)$ be an $(n+1)$-dimensional Lorentzian manifold with boundary such that $\Sigma =\partial M$ is a compact spacelike Cauchy hypersurface with $M=J^+(\Sigma)$ such that, along all normal geodesics $\gamma$ to the hypersurface $\Sigma$ we have the Ricci curvature bound 
\begin{align}\label{Ricbound}
\Ric(\dot{\gamma}, \dot{\gamma}) \ge 0.
\end{align}
Let $\mathcal{S}_t := \tau^{-1}({t})\setminus \mathrm{Cut}(\Sigma) =\Sigma_t\setminus  \mathrm{Cut}(\Sigma) $. Then 
\begin{align} \label{AreaEst}
|\mathcal{S}_t  |\leq \int_{\Sigma} \Big(\frac{H_+ t}{n}+1\Big)^n d\mu, 
\end{align}
where $H_+$ denotes the positive part of the mean curvature of $\Sigma$, i.e.,  $H_+(x)=\max\{H(x),0\}$ for $x\in \Sigma$.
\end{thm}
	
\begin{cor}\label{cor:volest}
	Under the same assumptions as the previous theorem we have
	\begin{align}\label{VolEst}
	|\Omega_t |=\int_0^t |\mathcal{S}_s|\, ds  \leq \int_{\Sigma} \int_0^t \Big(\frac{H_+ s}{n}+1\Big)^{n} ds\; d\mu.
	\end{align}
\end{cor}
\begin{proof} For $|\Omega_t \setminus \mathrm{Cut}(\Sigma)|$ this is just the coarea formula. The result for $|\Omega_t|$ follows because the cut locus has measure zero.
\end{proof}

	\br[Mean curvature bounds] \label{rem:meancurvrem} \normalfont Before proving Theorem~\ref{thm:areaest} we want to discuss several sufficient conditions on the mean curvature or second fundamental form to guarantee boundedness of $\int_{\Sigma} \Big(\frac{H_+ t}{n}+1\Big)^n d\mu$ and hence of $|\mS_{t}|$ and $|\Omega_t|$.

	{\em (i)}  Using Jensen's inequality, \eqref{AreaEst} becomes
	\begin{eqnarray}\label{eq:Jensenuse}
	\left|\mathcal{S}_t\right| &\leq& \int_{\Sigma} \left(\frac{H_+ t}{n}+1\right)^n d\mu
	\leq 2^{n-1} \int_{\Sigma} \left(\left(\frac{H_+ t}{n}\right)^n+1\right)d\mu\\
	&=& 2^{n-1} \left(\left(\frac{t}{n}\right)^n \int_{\Sigma} H_+^nd\mu+|\Sigma|\right). 
	\end{eqnarray} 
	So, as $\Sigma$ has itself finite area, bounding the $L^n$-norm of $H_+ $ is sufficient to get boundedness of the areas of $\mathcal{S}_t$. Further, since $H_+\leq |H|$ bounds on the $L^n$-norm of $H$ are sufficient as well and using this we may reformulate the conclusion of Theorem \ref{thm:areaest} and Corollary \ref{cor:volest} as
	\begin{align} \label{eq:constest}
	|\mathcal{S}_t  |\leq 
	2^{n-1}  \left(\left(\tfrac{t}{n}\right)^n \,||H||_{L^n}^n + {|\Sigma |} \right),
	\end{align}
	\begin{align} \label{eq:constest3}
	|\Omega_t |\leq 
	2^{n-1}  \left(\left(\tfrac{t^{n+1}}{n^n(n+1)}\right) \,||H||_{L^n}^n + t\,{|\Sigma |} \right).
	\end{align}
	
	Further, we could replace $||H||_{L^n}$ by $||H||_{L^p}$ for any $p>n$ in the estimates above as \be\int_{\Sigma} f^n d\mu \leq \big(\int_{\Sigma} f^p d\mu\big)^{\frac{n}{p}}\ee for all $p>n$ (again by Jensen's inequality). For $p<n$, however, this is will no longer work and Example \ref{ex:ex1} strongly indicates that (barring any additional assumptions) an $L^p$-bound on the mean curvature for any $p<n$ is insufficient for area and volume bounds of the form \eqref{eq:constest}, \eqref{eq:constest3}. \vspace{0.5em}

	{\em (ii)} Note further that an $L^n$ bound on the second fundamental form $K$ 
	implies an $L^n$ bound on $H=\mathrm{tr} (K)$: 
	
	Since $K$ is a $(0,2)$-tensor, $|K|^2=g^{im}g^{jl}K_{ij}K_{ml}$ so, for an orthonormal basis $e_1,\dots ,e_n \in T\Sigma $, we have $|K|^2=\sum_{i=0}^{n}\sum_{j=0}^{n} (K(e_i,e_j))^2$ and $|H|^2=(\sum_{j=1}^n K(e_j,e_j))^2$. Thus we estimate
	\be|H|^2\leq n^2 \max_{0\leq j\leq n} (K(e_j,e_j))^2 \leq n^2 |K|^2,\ee
	so $|H|^n\leq n^n \, |K|^n$ and so $||H||_{L^n}^n\leq n^n \, ||K||_{L^n}^n $.
	To summarize,
	\begin{cor}\label{cor:IDest} Under the same assumptions as in Theorem \ref{thm:areaest} we have
		\begin{align} \label{eq:constest1}
		|\mathcal{S}_t  |\leq 
		2^{n-1}  \left(t^n \, ||K||_{L^n}^n + {|\Sigma |} \right),
		\end{align}	\begin{align} \label{eq:constest4}
		|\Omega_t  |\leq 
		2^{n-1}  \left(\frac{t^{n+1}}{(n+1)} \, ||K||_{L^n}^n + t {|\Sigma |} \right).
		\end{align}
	\end{cor}
	\er

\br[Comparison of \eqref{AreaEst} to the estimates from \cite{TG}]\label{Comp-to-GT}

\normalfont Corollary~\ref{cor:IDest} (resp.~\eqref{eq:constest}, \eqref{eq:constest3}) shows that any sequence of Cauchy developments $(M_i, g_i)\supseteq (\Sigma_i , h_i, K_i)$ with uniform bounds on $|\Sigma_i |_{h_i}$ and $||K_i||_{L^n(\Sigma_i)}$ (resp.~$||H_i||_{L^n(\Sigma_i)}$) has uniformly bounded areas $|(\mathcal{S}_t)_i |_{g_i}$ and volumes  $|(\Omega_t)_i|_{g_i}$. 
 
 This is the main feature distinguishing Theorem \ref{thm:areaest} from earlier estimates available in  \cite{TG}: There, under the same assumptions as in Thm.~\ref{thm:areaest}, except for demanding the stronger {\em pointwise} mean curvature bound $H\leq \beta $ for some $\beta \in \mathbb{R}$,  Treude and Grant have shown
 \be\label{eq:TGAest}
|\mathcal{S}_t  |\leq |\Sigma |\; (\frac{\beta t}{n}+1)^n,
\ee
 \be\label{eq:TGVest}
|\Omega_t  |\leq |\Sigma |\,\int_0^t (\frac{\beta s}{n}+1)^n ds ,
\ee
for $0\leq t\leq \frac{n}{|\beta |}$ if $\beta <0$ and for all $t\geq 0$ if $\beta \geq 0$.

This only implies uniform boundedness of these areas if all $H_i$ are uniformly pointwise bounded which is a lot stronger than assuming a uniform bound on the  $L^n(\Sigma_i)$-norms of the $H_i$ --- while by compactness of $\Sigma$ each $H_i$ individually must be pointwise bounded by some $\beta_i$ this will generally not be uniform. Further, in Corollary \ref{cor:noncompact}, we extend \eqref{AreaEst} to non-compact $\Sigma $ with finite area, where even an individual mean curvature might not be pointwise bounded despite having finite $L^n$-norm.
\er

\br[Sharpness of estimates and rigidity]\normalfont  Note that the bounds in \eqref{eq:constest} and \eqref{eq:constest3} are not sharp due to the use of Jensen's inequality in the last step of the proof. In particular the model geometries defined in section \ref{sec:compspaces} will not satisfy \eqref{eq:constest} or \eqref{eq:constest3}  with equality. 

However, the original estimates \eqref{AreaEst} and \eqref{VolEst} of Theorem \ref{thm:areaest} and Corollary \ref{cor:volest} and their counterparts \eqref{subsetAreaEst} and \eqref{subsetVolEst} for $|\mS_t(A)|$ and $|\Omega_t(A)|$ of Lemma \ref{lem:subsets} are sharp in the following sense: Let $A \subseteq \mathbb{H}^n$ be compact. Then for any $\beta>0$ we consider the model geometry $M=[0,\infty)\times \mathbb{H}^n$ with metric $g_{\beta}=-dt^2+(t+\frac{n}{\beta})^2h_{-1}$
(cf.~section \ref{sec:compspaces}). Since $H=H_+=\beta >0$ is constant on $\mathbb{H}^n$, \eqref{subsetAreaEst}
becomes
\be |\mS_t(A)| 
\le \int_{A } \left(\frac{\beta t}{n}+1\right)^n d\mu,
\ee
which is equal to $|\mS_t(A)|=|A| \left(\frac{\beta t}{n}+1\right)^n$ (cf.~\eqref{eq:comparisonarea}).

We can also derive sharp inequalities in terms of the $L^p$-norms of $H_+$ for $p=1,\dots,  n$ from \eqref{AreaEst} by doing a binomial expansion. 
In particular, for $n=3$, we obtain
\begin{eqnarray}\label{sharpest1}
|\mathcal{S}_t| 
&\le & \left(\tfrac{t}{3}\right)^3 ||H_+||^3_{L^3}+ 3 \left(\tfrac{t}{3}\right)^2 ||H_+||^2_{L^2} 
+ 3 \left(\tfrac{t}{3}\right)||H_+||_{L^1} + |\Sigma|,
\end{eqnarray}
\begin{eqnarray}\label{sharpest2}
|\Omega_t| 
&\le & \tfrac{1}{4}\left(\tfrac{t}{3}\right)^4 ||H_+||^3_{L^3}+ \left(\tfrac{t}{3}\right)^3 ||H_+||^2_{L^2} 
+ \tfrac{3}{2}\left(\tfrac{t}{3}\right)^2||H_+||_{L^1} +  t\, |\Sigma| .
\end{eqnarray}
Regarding rigidity of those estimates note that we do not expect any strong rigidity: For example assuming $||H_+||_{L^3}\leq C_3, ||H_+||_{L^2} \leq C_2$ and $||H_+||_{L^1}\leq C_1$ for constants $C_i$ and that \[ 	|\mathcal{S}_t| 
= \left(\tfrac{t}{3}\right)^3 C_3^3+ 3 \left(\tfrac{t}{3}\right)^2 C_2^2
+ 3 \left(\tfrac{t}{3}\right)C_1 + |\Sigma| \] for all $t$ (i.e., that  \eqref{sharpest1} holds with equality for all $t$) would only allow us to conclude that $||H_+||_{L^3}= C_3, ||H_+||_{L^2} = C_2, ||H_+||_{L^1}= C_1$, which is not sufficient information to determine $H_+$ (one can easily imagine two different non-negative functions having the same $L^p$-norms for all $p$, e.g., by taking a non-constant function and just translating it).

On the other hand, we believe the following weak rigidity result requiring a {\em pointwise bound on $H$ by a continuous function} to hold, i.e., under the assumptions of Theorem \ref{thm:areaest}, if 
$0\leq H\leq f$ for a continuous $f:\Sigma \to \mathbb{R}$ and \be \label{eq:rgidityassumption}	|\mathcal{S}_t| 
= \int_{\Sigma} (\frac{f t}{n}+1)^n d\mu, \ee
then one should be able to conclude  $H=f$ on $\Sigma$ and $\Ric(\dot{\gamma},\dot{\gamma})=0$ for all future directed timelike geodesics $\gamma$ starting orthogonally to $\Sigma$.

This is in line with intermediate results in rigidity statements for the earlier estimates by Treude and Grant shown in \cite{Graf-splitting}. Note that we do not expect to recover any statement about isometry to a model space unless the upper bound function for $H$ is constant.  
\er

\subsection{Proof of Theorem \ref{thm:areaest}}

\begin{proof}[Proof of Theorem \ref{thm:areaest}]  For any $\epsilon >0$ cover $\Sigma$ by a finite number of disjoint measurable  $U_i\subseteq \Sigma$ with $\mathrm{diam}(U_i)\leq \epsilon$.\footnote{This can be achieved by, e.g., covering $\Sigma$ with balls of diameter $\epsilon $, extracting a finite subcover $B_i, i=1,\dots, N$ and defining $U_i:=B_i\setminus \bigcup^{i-1}_{j=0} B_j$.} 

Since the $U_i$ are disjoint, the sets $\mathcal{S}_t(U_i)$ are as well. Further, each  $\mS_t(U_i)$ is measurable 
 and they cover $ \mS_t$ since the $U_i$ cover $\Sigma$, so
\be|\mS_t|=\sum_{i=0}^N |\mS_t(U_i)|. \ee
We now use \cite[Thm.~8]{TG} to estimate $|\mS_t(U_i)|$, sketching their estimates here for completeness\footnote{Note also that while \cite{TG} assume that $H\leq \beta$ on all of $\Sigma$, as we will see their proof really only needs the bound at points $p_\Sigma\in\Sigma$ that are base points of maximizing geodesics from $U_i$ to $\mS_t(U_i)$.}:
Let 
\be
\beta_i:=\sup_{x\in U_i}H_+(x) =\max_{x\in \overline{U_i}}H_+(x) <\infty.
\ee

Then $H|_{U_i}\leq \beta_i$. Since $\tau =\tau_\Sigma$ is a timelike distance function on $\mM$ (\cite[Sec.~2.5]{TG}), the shape operator $\mathbf{S}: X\mapsto \nabla_X \mathbf{n}$, where $\mathbf{n}$ is the future unit normal vector field to the $\mS_t$'s, satisfies the Riccati equation (cf.~\cite[eq.~(5)]{TG})
\be \label{eq:Riccati}
\nabla_{\mathbf{n}}\mathbf{S}+\mathbf{S}^2+\mathrm{Riem}(.,\mathbf{n})\mathbf{n}=0.
\ee
So, since the unique unit-speed maximizing geodesic $\gamma: [0,t]\to \mM$ from
$p_\Sigma \in U_i$ to $p\in \mS_t(U_i)$ is an integral curve of $\mathbf{n}$, looking at \eqref{eq:Riccati} along $\gamma $ and taking the trace gives 
\be
\frac{d}{d\tau}H_{\mS_\tau}(\gamma(\tau))+\tr_{\mS_\tau} (\mathbf{S}^2)(\gamma(\tau))+\mathrm{Ric}(\dot{\gamma}(\tau),\dot{\gamma}(\tau))=0,
\ee
where $H_{\mS_\tau}(\gamma(\tau))=\tr_{\mS_\tau} \mathbf{S}(\gamma(\tau))$ is the mean curvature of $\mS_\tau$ at $\gamma(\tau)$. Using non-negativity of $\mathrm{Ric}(\dot{\gamma}(\tau),\dot{\gamma}(\tau))$ and $(\tr(\mathbf{S}))^2\leq (n-1)\tr (\mathbf{S}^2)$ gives us
\be
\frac{d}{d\tau}H_{\mS_\tau}(\gamma(\tau))+\frac{1}{n-1}H_{\mS_\tau}(\gamma(\tau))\leq 0
\ee
So using $H_\Sigma(p_\Sigma)=H_{\mS_\tau}(\gamma(\tau))|_{\tau=0}\leq \beta_i$ a standard analysis of the Raychaudhuri equation implies
\be H_{\mS_\tau}(\gamma(\tau))\leq \frac{n}{\tau +\frac{n}{\beta_i}} \quad \forall \tau \leq t.
\ee
Note that while we don't necessarily need the comparison spacetimes discussed in section \ref{sec:compspaces} for our analysis the right hand side here actually equals the constant mean curvature of a $\{t=\tau\}$-slice in the comparison spacetime $M_{\beta_i}$.

Now let $0\leq t_1<t_2$ and let $K_{t_2}\subseteq \mS_{t_2}(U_i)$ be a compact subset and for $0\leq \tau \leq t_2$ denote by $K_\tau\subseteq \mS_\tau(U_i)$ its backwards flow along $\mathbf{n}$ at time $\tau$. We can use the variation of area formula  to estimate
\be
\frac{d}{d\tau}\log(|K_\tau|)=\frac{1}{|K_\tau |}\int_{K_\tau} H_{\mS_\tau}(q) d\mu_\tau(q)\leq \frac{n}{\tau +\frac{n}{\beta_i}}=\frac{d}{d\tau}\log((\frac{\beta_i \tau }{n}+1)^n).
\ee
So the quotient of $|K_\tau|$ and $(\frac{\beta_i \tau}{n}+1)^n$ is non-increasing in $\tau $, so
\be \frac{|K_{t_2}|}{(\frac{\beta_i t_2}{n}+1)^n}\leq \frac{|K_{t_1}|}{(\frac{\beta_i t_1}{n}+1)^n}\leq \frac{|\mS_{t_1}(U_i)|}{(\frac{\beta_i t_1}{n}+1)^n}.
\ee
Because this holds for all compact $K_{t_2}\subseteq \mS_{t_2}(U_i)$ it follows
that 
\be \label{eq:nonincreasingU_i}  \frac{|\mS_{t_2}(U_i)|}{(\frac{\beta_i t_2}{n}+1)^n}\leq \frac{|\mS_{t_1}(U_i)|}{(\frac{\beta_i t_1}{n}+1)^n}
\ee
and because $t_1<t_2$ were arbitrary we get that
\be t\mapsto \frac{|\mS_{t}(U_i)|}{(\frac{\beta_i t}{n}+1)^n}
\ee is non-increasing. Again this could be rewritten using the model geometries from section \ref{sec:compspaces} as
\be  t\mapsto \frac{|\mS_{t}(U_i)|}{|\Sigma_{\beta_i,t}(A)|}\quad \mathrm{is\; non-increasing}.
\ee

It follows that
\be \label{eq:areaestU_i}	|\mS_t(U_i)|\leq |U_i| \left(\frac{\beta_i t}{n}+1\right)^n=\int_{U_i} \left(\frac{\beta_i t}{n}+1\right)^n\,d\mu.
\ee

	Now we paste these local area estimates together to obtain an estimate for $|\mS_t|=\sum |\mS_t(U_i)|$: Let $\chi_i $ be the characteristic function of $U_i$ and define $f_\epsilon = \sum_{i=1}^N \chi_i\; (\frac{\beta_i t}{n}+1)^n$ (where we reintroduce $\epsilon \geq \diam(U_i)$ in the notation). Then  
	\be
	\left|\mathcal{S}_t\right|\leq \sum_i \int_{U_i} \left(\frac{\beta_i t}{n}+1\right)^n \,d\mu
	=\int_{\Sigma} f_\epsilon \,d\mu
	\ee
	for all $\epsilon$. By construction 
\be
\left|f_{\epsilon }(x)-\left(\frac{H_+(x) t}{n}+1\right)^n\right| \leq L\,\epsilon,
\ee	
	where $L$ is a global Lipschitz constant for $x\mapsto (\frac{H_+(x) t}{n}+1)^n$. So 
	\be
	f_\epsilon \to \left(\frac{H_+ t}{n}+1\right)^n \textrm{ uniformly}
	\ee
	 and hence  
	 \be
	 \int_{\Sigma} f_\epsilon\,d\mu\to \int_{\Sigma} \left(\frac{H_+ t}{n}+1\right)^nd\mu
	  \textrm{ as }\epsilon \to 0.
	  \ee 
So, 
	\begin{eqnarray}
	\left|\mathcal{S}_t\right| &\leq& \int_{\Sigma} \left(\frac{H_+ t}{n}+1\right)^n d\mu.
	\end{eqnarray}  \qed
\end{proof}

\section{Generalized area estimates for $\Sigma_t$}\label{sec:limitareas}

Let $\Sigma_t=\tau^{-1}(t)$ and recall 
\be
\mathcal{S}_t = \tau^{-1}({t})\setminus \mathrm{Cut}(\Sigma) =\Sigma_t\setminus  \mathrm{Cut}(\Sigma).
\ee
Define a generalized area for $\Sigma_t$ as
\be \label{eq:defgenArea}
a(\Sigma_t):=\limsup_{h \to 0+} \frac{1}{h}  |\tau^{-1}([t-h, t])|=\limsup_{h \to 0+} \frac{1}{h} (|\Omega_{t-h}|-|\Omega_t|).
\ee

\begin{lem}\label{Vol-Sigma-T-1} The generalized areas satisfy
\be
|\mathcal{S}_T| \le a(\Sigma_T) \le \limsup_{t\to T-} |\mathcal{S}_t|.
\ee
\end{lem}

\begin{proof}
We first show
\be
|\mathcal{S}_T| \le a(\Sigma_T). \label{eq:Vol-Sigma-T-1-ineq1}
\ee
Let $A\subseteq \Sigma $ be the set of all $p$ such that the unit speed geodesic normal to $\Sigma $ starting at $p$ intersects $\mS_T$. Then $\mS_T=\mS_T(A)$. Since any such geodesic can't encounter a cut point before or at $T$, they all maximize the distance to $\Sigma $ for all $t\leq T$. This implies that $t\mapsto |\mS_t(A)|$ is continuous on $[0,T]$ and hence by the coarea formula
\begin{align}
|\mS_T(A)|&=\lim_{h\to 0+} \frac{1}{h}\int_{T-h}^T |\mS_t(A)| dt\\
&\leq \limsup_{h\to 0+} \frac{1}{h}  \int_{T-h}^T |\mS_t| dt = \limsup_{h\to 0+} |\tau^{-1}([T-h, T])|=a(\Sigma_T)
\end{align}
and we have shown \eqref{eq:Vol-Sigma-T-1-ineq1}. 
To show
\be
a(\Sigma_T)\le \limsup_{t\to T-} |\mathcal{S}_t|,
\ee
assume on the contrary that $\exists \epsilon>0$ such that
\be 
a(\Sigma_T) > \limsup_{t\to T-} |\mathcal{S}_t| + \epsilon.
\ee
Thus there exists $\delta>0$ such that 
\be 
a(\Sigma_T) >|\mathcal{S}_t| + \epsilon/2 \qquad \forall t\in (T-\delta,T).
\ee
Integrating $t$ from $T-h$ to $T$ for $h<\delta$ we have
\begin{eqnarray}
a(\Sigma_T) &>& \frac{1}{h} \int_{T-h}^T |\mathcal{S}_t| \, dt \,\,+ \,\,\epsilon/2  \\
&=& \frac{1}{h} |\tau^{-1}({T-h,T})\setminus \mathrm{Cut}(\Sigma)|\,\,+ \,\,\epsilon/2 \\
&=& \frac{1}{h} |\tau^{-1}({T-h,T})|\,\,+ \,\,\epsilon/2 
\end{eqnarray}
by the coarea formula and the fact that the cut locus has $(n+1)$-dimensional measure zero.
Taking the limsup as $h\to 0+$ we get a contradiction. \qed
\end{proof}

From this and \eqref{eq:constest} we then immediately have:

\begin{cor}\label{Vol-Sigma-T-2}
Under the hypotheses of Theorem~\ref{thm:areaest} 
\be
a(\Sigma_t) \le  \int_{\Sigma} \Big(\frac{H_+ t}{n}+1\Big)^n d\mu \leq 2^{n-1}  \left(\left(\tfrac{t}{n}\right)^n \,||H||_{L^n}^n + {|\Sigma |} \right)
\ee
for all $t\geq 0$.
\end{cor}

\section{Extending Theorem \ref{thm:areaest} to subsets and non-compact $\Sigma $ with finite area}

Looking at the proof of Theorem \ref{thm:areaest} we see that, even if $\Sigma $ itself were non-compact, the same estimates as in Theorem \ref{thm:areaest} (and the subsequent Remark \ref{rem:meancurvrem}) apply to the time evolution of the area of relatively compact subsets of $\Sigma $.

\begin{lem}\label{lem:subsets}
	Let $(M, g)$ be an $(n+1)$-dimensional Lorentzian manifold with boundary and $\Sigma = \partial M$ a possibly non-compact spacelike Cauchy hypersurface such that $M=J^+(\Sigma)$. Let $A\subseteq \Sigma $ be   relatively compact and measurable such that along all normal geodesics $\gamma$ to the hypersurface $\Sigma$ starting in $A$  we have the Ricci curvature bound 
	\begin{align}
	\Ric(\dot{\gamma}, \dot{\gamma}) \ge 0.
	\end{align}
	Then 
	\begin{align} \label{subsetAreaEst}
	|\mathcal{S}_t(A)  |\leq \int_{A} \Big(\frac{H_+ t}{n}+1\Big)^{n} d\mu
	\end{align}
	and
	\begin{align}\label{subsetVolEst}
	|\Omega_t(A) | = \int_0^t |\mathcal{S}_s(A)|\, ds  \leq \int_{A} \int_0^t \Big(\frac{H_+ s}{n}+1\Big)^{n} ds\; d\mu.
	\end{align} 
\end{lem}

Lemma \ref{lem:subsets} can also be used to estimate volumes and areas in domains of outer communications:

\begin{cor}\label{cor:DOC} Let $D\subseteq M$ be a past set\footnote{A past set is a set $D\subseteq M$ such that $I^-(D)\subseteq D$.} such that $\Sigma \cap D$ is compact. Under the assumptions of Theorem \ref{thm:areaest} (except for allowing $\Sigma $ being non-compact), we have
	\begin{align} \label{AreaEstD}
	|\mathcal{S}_t \cap D |\leq \int_{\Sigma \cap D} \Big(\frac{H_+ t}{n}+1\Big)^{n} d\mu 
	\end{align}
	and
	\begin{align}\label{VolEstD}
	|\Omega_t \cap D| \leq \int_{\Sigma \cap D} \int_0^t \Big(\frac{H_+ s}{n}+1\Big)^{n} ds\; d\mu.
	\end{align}
	In particular for a black hole spacetime we can then bound  $|\mathcal{S}_t \cap \mathcal{D}|$ and $|\Omega_t \cap \mathcal{D}|$, where $\mathcal{D}$ denotes the domain of outer communications, in terms of bounds on the area and the integral of $H_+^n$ over the part of $\Sigma $ outside all horizons only. 
\end{cor}
\begin{proof}
	We only need to show that $\mathcal{S}_t \cap D\subseteq \mathcal{S}_t(\Sigma\cap D)$, then the claims follow from Lemma \ref{lem:subsets}. Let $x \in \mathcal{S}_t \cap D$. Then there exists a unique future directed maximizing timelike geodesic $\gamma$ from $\Sigma$ to $x$ which does not meet the cut locus. Since $D$ is a past set, $\gamma\subseteq I^-(x)\subseteq D$, so $\gamma$ starts in $\Sigma \cap D$ and thus $x\in \mathcal{S}_t(\Sigma\cap D)$.
\end{proof}

Using \ref{lem:subsets} we can also show estimates \eqref{AreaEst} and \eqref{VolEst} for non-compact $\Sigma$ with finite area.

\begin{cor}\label{cor:noncompact}
Estimates  \eqref{AreaEst} and \eqref{VolEst} still hold even if the Cauchy hypersurface in Theorem \ref{thm:areaest} and Corollary \ref{cor:volest} is non-compact (as long as all other assumptions, in particular \eqref{Ricbound}, are still satisfied). 
\end{cor}
\begin{proof}
Let $K_i$ be an exhaustion by compact sets of $\Sigma$. Then, 
$\mathcal{S}_t(K_i)\subseteq \mathcal{S}_t(K_{i+1})$ and $\bigcup_{i=1}^\infty \mathcal{S}_t(K_i)= \mathcal{S}_t$. By Lemma \ref{lem:subsets} we have
\begin{align*}
|\mathcal{S}_t(K_i)| &\leq \int_{K_i} \Big(\frac{H_+ t}{n}+1\Big)^{n} d\mu \leq  \int_{\Sigma } \Big(\frac{H_+ t}{n}+1\Big)^{n} d\mu .
\end{align*}
So 
\be|\mathcal{S}_t|= |\bigcup_{i=1}^\infty \mathcal{S}_t(K_i)|=\lim_{i\to \infty} |\mathcal{S}_t(K_i)| \leq  \int_{\Sigma } \Big(\frac{H_+ t}{n}+1\Big)^{n} d\mu \ee
The estimate for the volume follows again from the coarea formula and integration.
\end{proof}

\br \label{rem:noncompactA} Using the same argument we see that requiring $A$ to be relatively compact in Lemma \ref{lem:subsets} was unnecessary and that it suffices that $A$ has finite area.
\er

 As a direct consequence of Corollary~\ref{cor:noncompact}, we obtain the following corollary.
 
\bc  If, under the assumptions of Theorem \ref{thm:areaest},  $\Sigma $ is non-compact but has finite area and $||H_+||_{L^n(\Sigma)}$ is bounded, then  the areas $|\mS_t |$ and volumes $|\Omega_t|$ remain bounded.
\ec

\section{Example: For $p<n$, bounds on the $L^p$-norm of $H_+$ are insufficient for the estimates \eqref{eq:constest}, \eqref{eq:constest3} } \label{Sect-Example}

In this last section we want to illustrate that bounds on the $L^p$-norm of $H$ are not sufficient for volume bounds analogous to \eqref{eq:constest} or \eqref{eq:constest3}. Specifically, our examples will contradict analogues to the "subset versions" of \eqref{eq:constest} or \eqref{eq:constest3} based on Lemma \ref{lem:subsets}: For any $p<n$ we will give a family of spacetime metrics $g_j$ on $M=[0,\infty )\times \mathbb{H}^n$ 
satisfying the assumptions of Lemma~\ref{lem:subsets} for compact $A_j\subseteq \mathbb{H}^n$ for which the $L^p(A_j)$-norms of $H_j$ remain uniformly bounded but for which $|\mS_t(A_j)|_{g_j}\to \infty$.

\begin{ex}\label{ex:ex1}\normalfont 
Fix $p<n$. Let $(\mathbb{H}^n,h_{-1})$ be standard hyperbolic space and fix two disjoint open relatively compact subsets $B_1,B_2\subseteq \mathbb{H}^n$ with $|B_1|,|B_2|\geq 1$. For $j\in \N $ let $A_{1,j}\subseteq B_1, A_{2,j}\subseteq B_2$ be compact with
\be |A_{1,j}|_{h_{-1}}=\frac{1}{2j} ,\qquad |A_{2,j}|_{h_{-1}}=1-\frac{1}{2j}
\ee
Set 
\begin{align}\beta_{1,j}&:=j^{\frac{1}{p}}\\
\beta_{2,j}&:=\big(\frac{1}{2-\frac{1}{j}}\big)^{\frac{1}{p}}
\end{align}
 and let $f_j:[0,\infty)\times \mathbb{H}^n \to (0,\infty)$ be a smooth function satisfying 
\begin{align} f_j(t,x)&=(t+\frac{n}{\beta_{1,j}}) \quad &\mathrm{for}\quad (t,x) \in [0,\infty)\times B_1 \\
 f_j(t,x)&=(t+\frac{n}{\beta_{2,j}}) \quad &\mathrm{for}\quad (t,x) \in [0,\infty)\times B_2
\end{align}
and let
\be
g_j:= -dt^2+f_j(t,x)^2 h_{-1}.
\ee 

Denoting the mean curvature of $\Sigma_j:=\{0\}\times \mathbb{H}^n\subseteq ([0,\infty )\times\mathbb{H}^n, g_j)   $ by $H_j$ we note that 
\begin{align}H_j|_{A_{1,j}}&=\beta_{1,j}=j^{\frac{1}{p}}>0\\
H_j|_{A_{2,j}}&= \beta_{2,j}=\big(\frac{1}{2-\frac{1}{j}}\big)^{\frac{1}{p}}>0.
\end{align}

Since normal geodesics to $\Sigma_j$ starting in $A_{1,j}$ (resp.~$A_{2,j}$) never leave $[0,\infty)\times B_1$ (resp.~$[0,\infty)\times B_2$) where $g_j$ is one of the model geometry metrics of Section~\ref{sec:compspaces}, we have
\[\Ric(\dot{\gamma},\dot{\gamma})=0,\]
i.e., the assumptions of Lemma~\ref{lem:subsets} are satisfied for $A_j:=A_{1,j}\cup A_{2,j}$ and hence 
by \eqref{eq:comparisonarea} we have
\begin{align}
|\mS_t(A_j)|&= |\mS_t(A_{1,j})|+ |\mS_t(A_{2,j})|\\
&=|A_{1,j}| \left( 1+ \frac{\beta_{1,j} t}{n} \right)^n+|A_{2,j}| \left( 1+ \frac{\beta_{2,j} t}{n} \right)^n \\
&= \frac{1}{2j}\Big(\frac{j^{\frac{1}{p}}t}{n}+1\Big)^n + \big(1-\frac{1}{2j}\big) \Big(\frac{t}{(2-\frac{1}{j})^{\frac{1}{p}}n}+1\Big)^n \to \infty
\end{align}
as $j\to \infty $. Note that 
\be ||H_j||_{L^p(A_j)}=\big(j|A_{1,j}|_{h_{-1}}+\frac{1}{2-\frac{1}{j}} |A_{2,j}|_{h_{-1}} \big)^\frac{1}{p}=1 , \ee
so the above in particular implies 
\be|\mS_t(A_j)| \not \leq   2^{n-1}   \left(\left(\tfrac{t}{n}\right)^n \,||H_j||_{L^p(A_j)}^n + {|\Sigma |} \right)\ee
for large $j$, so the "subset version" of \eqref{eq:constest} can't hold. The same immediately follows for the "subset version" of \eqref{eq:constest3}. 
\end{ex}

\begin{rmrk} \label{L^2-L^n}
In work of Bryden, Khuri, and the second author \cite{BKS21}, a theorem is proven showing that if the mass of a spherically symmetric sequence of initial data sets $(\Sigma_j, h_j, k_j)$ converges to $0$ then there are $L^p$ bounds on the second fundamental form for $p \in [1,2)$.  Thus one might want to obtain volume estimates in such a setting.   However the example above indicates that our approach to estimating the volume requires the stronger $L^n$ bound.  This would be interesting to study further perhaps requiring that the regions solve an Einstein Equation with a particular matter model.
\end{rmrk}

\bibliographystyle{plain}
\bibliography{Graf-Sormani-20}

\end{document}